\documentclass{amsart}

\usepackage{graphicx}
\usepackage{amsfonts}
\usepackage{amsmath}
\usepackage{amscd}
\usepackage{amssymb}
\usepackage{verbatim}
\usepackage{hyperref}

\newenvironment{proof*}{\noindent\emph{Proof}}{$\square$\smallskip}

\newtheorem{theorem}{Theorem}[section]

\newtheorem{Definition}[theorem]{Definition}
\newtheorem{lemma}[theorem]{Lemma}
\newtheorem{Example}[theorem]{Example}

\newtheorem{Remark}[theorem]{Remark}
\newtheorem{proposition}[theorem]{Proposition}

\newtheorem{Exercise}[theorem]{Exercise}
\newtheorem{Exercises}[theorem]{Exercises}
\newtheorem{Notation}[theorem]{Notation}
\newtheorem{Convention}[theorem]{Convention}

\newenvironment{definition}{\begin{Definition}\normalfont}{\end{Definition}}
\newenvironment{example}{\begin{Example}\normalfont}{\end{Example}}
\newenvironment{remark}{\begin{Remark}\normalfont}{\end{Remark}}

\newenvironment{convention}{\begin{Convention}\normalfont}{\end{Convention}}

\newcommand{\id}{\ensuremath{\mathrm{id}}} 

\newcommand{\co}{\ensuremath{\colon}} 

\newcommand{\sm}{\ensuremath{{\rm Sim}}} 

\title[Braided Diagram Groups]{Braided diagram groups and local similarity groups}  
\author[D.~S.~Farley]{Daniel S. Farley}
\address{Department of Mathematics and Statistics\\ Miami University\\ Oxford, OH 45056 U.S.A.}
\email{farleyds@muohio.edu}

\author[B.~Hughes]{Bruce Hughes}
\address{Department of Mathematics\\ Vanderbilt University\\ Nashville, TN 37240 U.S.A.}
\email{bruce.hughes@vanderbilt.edu}
\thanks{The second-named author was supported in part by NSF Grant DMS--0504176.}

\date{\today}

\begin{document}

\begin{abstract} 
Hughes defined a class of groups that act as local similarities on compact ultrametric spaces. Guba and Sapir had previously defined braided diagram groups over semigroup presentations. The two classes of groups share some common characteristics: both act properly by isometries on CAT(0) cubical complexes, and certain 
groups in both classes have type $F_{\infty}$, for instance.

Here we clarify the relationship between these families of groups: the braided diagram groups over tree-like semigroup presentations are precisely 
the groups that act on compact ultrametric spaces via  small similarity structures. The proof can be considered a generalization of the proof that Thompson's group $V$ 
is a braided diagram group over a tree-like semigroup presentation. 

We also prove that certain additional groups, such as the Houghton groups
$H_{n}$, and $QAut(T_{2,c})$, lie in both classes. 
\end{abstract}

\maketitle

\setcounter{tocdepth}{1} \tableofcontents

{\footnotesize {\hspace*{15pt}}2000 Math.\ Subject Class. Primary 20F65, 54E45 {\hspace*{30pt}} }
\tableofcontents

\section{Introduction}

In \cite{HugLSHP}, Hughes described a class of groups that act as homeomorphisms on compact ultrametric spaces. Fix a compact ultrametric space $X$.
The essence of the idea
was to associate to $X$ a \emph{finite similarity structure}, which is a function that associates to each ordered pair of balls $B_{1}, B_{2} \subseteq X$ a finite set $\sm_{X}(B_{1}, B_{2})$ of surjective similarities from $B_{1}$ to $B_{2}$. (A \emph{similarity} is a map that stretches or contracts distances by a fixed constant.) The finite
sets $\sm_{X}(B_{1}, B_{2})$ are assumed to have certain desirable closure properties (such as closure under composition). A homeomorphism $h: X \rightarrow X$ is said to be \emph{locally determined by $\sm_{X}$} if each $x \in X$ has a ball neighborhood $B$ with the property that $h(B)$
is a ball and 
the restriction of $h$ to $B$ agrees with one of the local similarities $\sigma \in \sm_{X}(B, h(B))$. The collection of all homeomorphisms that are locally determined by
$\sm_{X}$ forms a group under composition. We will call such a group an \emph{FSS group} (finite similarity structure group) for short. Hughes \cite{HugLSHP} proved
that each FSS group has the Haagerup property, and even acts properly on a CAT(0) cubical complex. In \cite{FH}, the authors described a class of FSS groups that
have type $F_{\infty}$. That class includes Thompson's group $V$, and the main theorem of \cite{FH} is best understood as a generalization of \cite{Brown1987},
where Brown originally showed that $V$ has type $F_{\infty}$.

In earlier work, Guba and Sapir \cite{GubaSapir} had sketched a theory of braided diagram groups over semigroup presentations, and proved that Thompson's group $V$
is a braided diagram group over the semigroup presentation $\langle x\mid x= x^{2} \rangle$. Farley \cite{Far3} showed that braided diagram groups over semigroup presentations
act properly on CAT(0) cubical complexes.

The class $\mathcal{F}$ of FSS groups and the class $\mathcal{B}$ of braided diagram groups therefore have a common origin, as generalizations of Thompson's group $V$. Both classes also share other features in common (as noted above). It is therefore natural to wonder to what extent the two classes are the same. The main goal of this note is to prove Theorem \ref{thm: bigisomorphism}, which says that the FSS groups determined by small similarity structures (Definition \ref{def: smallsim})
are precisely the same as the braided diagram groups determined by tree-like semigroup presentations (Definition \ref{def: treelike}). It is even possible that Theorem \ref{thm: bigisomorphism} describes the precise
extent of the overlap between $\mathcal{F}$ and $\mathcal{B}$, but we do not know how to prove this.

We include all relevant definitions, and our treatment is fairly self-contained as a result. A precise definition of braided diagram groups is given in 
Section \ref{sec:diagramgroups}, the precise definition of FSS groups appears in Section \ref{section:fss}, and the main theorem is proved in Section
\ref{section:main}. Along the way, we give additional examples in the class $\mathcal{F} \cap \mathcal{B}$, including the Houghton groups $H_{n}$ and
a certain group $QAut(T_{2,c})$ of quasi-automorphisms of the infinite binary tree. (These are Examples \ref{example:houghton} and \ref{example:Q}, respectively.)

This note has been adapted from the longer preprint \cite{FH}. The first part of the latter preprint (including roughly the first six sections) will be published elsewhere. The first author would like to thank the organizers of the Durham Symposium (August 2013) for the opportunity to speak. Example \ref{example:houghton} first appeared as part of the 
first author's lecture. The idea of Example \ref{example:Q} occurred to the first author after listening to Collin Bleak's lecture at the Symposium.

\section{Braided diagram groups} \label{sec:diagramgroups}

In this section, we will recall the definition of braided diagram groups over semigroup presentations. 
Note that the theory of braided diagram groups was first sketched by Guba and Sapir \cite{GubaSapir}. A more extended introduction to braided diagram groups
appears in \cite{Far3}.

\begin{definition}
Let $\Sigma$ be a set, called an \emph{alphabet}. The \emph{free semigroup on $\Sigma$}, denoted $\Sigma^{+}$, is
the collection of all positive non-empty strings formed from $\Sigma$, i.e., 
$$ \Sigma^{+} = \{ u_{1}u_{2} \ldots u_{n} \mid n \in \mathbb{N}, u_{i} \in \Sigma \text{ for } i \in \{ 1, \ldots, n \} \}.$$
The \emph{free monoid on $\Sigma$}, denoted $\Sigma^{\ast}$, is the union $\Sigma^{+} \cup \{ 1 \}$, where $1$ denotes the
empty string. (Here we assume that $1 \not \in \Sigma$ to avoid ambiguity.) The operations in $\Sigma^{+}$ and $\Sigma^{\ast}$ are concatenation.

We write $w_{1} \equiv w_{2}$ if $w_{1}$ and $w_{2}$ are equal as words in $\Sigma^{\ast}$.

\end{definition}

\begin{definition}
A \emph{semigroup presentation} $\mathcal{P} = \langle \Sigma \mid \mathcal{R} \rangle$ consists of an alphabet $\Sigma$ and
a set $\mathcal{R} \subseteq \Sigma^{+} \times \Sigma^{+}$. The elements of $\mathcal{R}$ are called \emph{relations}.
\end{definition}

\begin{remark}
A relation $(w_{1}, w_{2}) \in \mathcal{R}$ can be viewed as an equality between the words $w_{1}$ and $w_{2}$. 
We use ordered pairs to describe these equalities because we will occasionally want to make a distinction between the 
left and right sides of a relation.

A semigroup presentation $\mathcal{P}$ determines a semigroup $S_{\mathcal{P}}$, just as a group presentation determines a group. We will, however, 
make essentially no use of this semigroup $S_{\mathcal{P}}$. Our interest is in braided diagrams over $\mathcal{P}$ (see below).
\end{remark}

\begin{definition} \label{def: braideddiagrams}
(Braided Semigroup Diagrams)
A \emph{frame} is a homeomorphic copy of $\partial([0,1]^{2}) = ( \{ 0, 1 \} \times [0,1]) \cup ( [0,1] \times \{ 0, 1 \} )$. A frame
has a \emph{top} side, $(0,1) \times \{ 1 \}$, a \emph{bottom} side, $(0,1) \times \{ 0 \}$, and \emph{left} and \emph{right} sides,
$\{ 0 \} \times [0,1]$ and $\{ 1 \} \times [0,1]$, respectively. The top and bottom of a frame have obvious left to right orderings.

A \emph{transistor} is a homeomorphic copy of $[0,1]^{2}$. A transistor has top, bottom, left, and right sides, just as a frame does. The
top and bottom of a transistor also have obvious left to right orderings.

A \emph{wire} is a homeomorphic copy of $[0,1]$. Each wire has a bottom $0$ and a top $1$.

Let $\mathcal{P} = \langle \Sigma \mid \mathcal{R} \rangle$ be a semigroup presentation. Let $\mathcal{T}(\Delta)$ be a finite
(possibly empty) set of transistors. Let $\mathcal{W}(\Delta)$ be a finite, nonempty  set of wires. We let $F(\Delta) = \partial ( [0,1]^{2})$ be
a frame. We let $\ell_{\Delta}: \mathcal{W}(\Delta) \rightarrow \Sigma$ be an arbitrary function, called the \emph{labelling function}.

For each wire $W \in \mathcal{W}(\Delta)$, we choose a point $t(W)$ on the bottom of a transistor, or on the top of the frame, and a point $b(W)$
on the top of a transistor, or on the bottom of the frame. The points $t(W)$ and $b(W)$ are called the \emph{top} and \emph{bottom contacts} of $W$, respectively.

We attach the top of each wire $W$ to $t(W)$ and the bottom of $W$ to $b(W)$. The resulting topological space $\Delta$ is called a \emph{braided diagram over
$\mathcal{P}$} if the following additional conditions are satisfied:

\begin{enumerate}
\item If $W_{i}$, $W_{j} \in \mathcal{W}(\Delta)$, $t(W_{i}) = t(W_{j})$ only if $W_{i} = W_{j}$, and $b(W_{i}) = b(W_{j})$ only if $W_{i} = W_{j}$.
In other words, the disjoint union of all of the wires maps injectively into the quotient.
\item We consider the top of some transistor $T \in \mathcal{T}(\Delta)$. Reading from left to right, we find contacts 
$$ b(W_{i_{1}}), b(W_{i_{2}}), \ldots, b(W_{i_{n}}),$$
where $n \geq 0$. The word $\ell_{t}(T) = \ell(W_{i_{1}}) \ell(W_{i_{2}}) \ldots \ell(W_{i_{n}})$ is called the \emph{top label of $T$}.
Similarly, reading from left to right along the bottom of $T$, we find contacts
$$ t(W_{j_{1}}), t(W_{j_{2}}), \ldots, t(W_{j_{m}}),$$
where $m \geq 0$. The word $\ell_{b}(T) = \ell(W_{j_{1}}) \ell(W_{j_{2}}) \ldots \ell(W_{j_{m}})$ is called the \emph{bottom label of $T$}.
We require that, for any $T \in \mathcal{T}(\Delta)$, either $(\ell_{t}(T), \ell_{b}(T)) \in \mathcal{R}$ or $(\ell_{b}(T), \ell_{t}(T)) \in \mathcal{R}$.
(We emphasize that it is not sufficient for $\ell_{t}(T)$ to be equivalent to $\ell_{b}(T)$ modulo the relation $\sim$ determined by $\mathcal{R}$. Note also that this condition implies that
$\ell_{b}(T)$ and $\ell_{t}(T)$ are both non-empty, since $\mathcal{P}$ is a semigroup presentation. In particular, each transistor has wires attached to its top and bottom faces.)
\item We define a relation $\preceq$ on $\mathcal{T}(\Delta)$ as follows. Write $T_{1} \preceq T_{2}$ if there is some wire $W$ such that
$t(W) \in T_{2}$ and $b(W) \in T_{1}$. We require that the transitive closure $\dot{\preceq}$ of $\preceq$ be a strict partial order on $\mathcal{T}(\Delta)$.
\end{enumerate}
\end{definition}

\begin{definition} \label{def: (a,b)diagrams}
Let $\Delta$ be a braided diagram over $\mathcal{P}$. Reading from left to right across the top of the frame $F(\Delta)$,
we find contacts 
$$ t(W_{i_{1}}), t(W_{i_{2}}), \ldots, t(W_{i_{n}}),$$
for some $n \geq 1$. The word $\ell(W_{i_{1}})\ell(W_{i_{2}}) \ldots \ell(W_{i_{n}}) = \ell_{t}(\Delta)$
is called the \emph{top label of $\Delta$}. We can similarly define the \emph{bottom label of $\Delta$}, $\ell_{b}(\Delta)$.
We say that $\Delta$ is a \emph{braided $(\ell_{t}(\Delta), \ell_{b}(\Delta))$-diagram over $\mathcal{P}$}.
\end{definition}

\begin{remark} \label{rem: caution} 
One should note that braided diagrams, despite the name, are not truly braided. In fact, two braided diagrams are equivalent
(see Definition \ref{def: equivalentdiagrams}) if there is a certain type of marked homeomorphism between them. Equivalence therefore does not depend
on any embedding into a larger space. Braided diagram groups (as defined in Theorem \ref{thm: braideddiagramgroups}) also seem to have little in common 
with Artin's braid groups.  
\end{remark}

\begin{example}
Let $\mathcal{P} = \langle a, b, c \mid ab=ba, ac=ca, bc=cb \rangle$. Figure \ref{figure1} shows an example of a braided $(aabc, acba)$-diagram
over the semigroup presentation $\mathcal{P}$. The frame is the box formed by the dashed line. The wires that appear to cross in the figure do not
really touch, and it is unnecessary to specify which wire passes over the other one. See Remark \ref{rem: caution}.
\begin{center}
\begin{figure}[!h]
\includegraphics{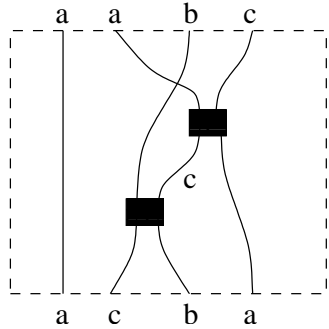}
\caption{A braided $(aabc, acba)$-diagram over the semigroup presentation $\mathcal{P} = \langle a, b, c \mid ac=ca, ab=ba, bc=cb \rangle$.}
\label{figure1}
\end{figure}
\end{center}
\end{example}

\begin{definition}
(Concatenation of braided diagrams)
Let $\Delta_{1}$ and $\Delta_{2}$ be braided diagrams over $\mathcal{P}$. We suppose that $\Delta_{1}$ 
is a $(w_{1}, w_{2})$-diagram and $\Delta_{2}$ is a $(w_{2},w_{3})$-diagram. We can multiply $\Delta_{1}$ and $\Delta_{2}$ by stacking them. More explicitly,
we remove the bottom of the frame of $\Delta_{1}$ and the top of the frame of $\Delta_{2}$, and then glue together the wires in order from left to right. This gluing is compatible with the labeling of the wires, since the bottom label of $\Delta_{1}$ is the same as the top label of $\Delta_{2}$. The result
is a braided diagram $\Delta_{1} \circ \Delta_{2}$, called the \emph{concatenation} of $\Delta_{1}$ and $\Delta_{2}$. 
\end{definition}

\begin{definition}
(Dipoles)
Let $\Delta$ be a braided semigroup diagram over $\mathcal{P}$. We say that the transistors $T_{1}, T_{2} \in \mathcal{T}(\Delta)$, $T_{1} \, \preceq \, T_{2}$,
form a \emph{dipole} if:
\begin{enumerate}
\item the bottom label of $T_{1}$ is the same as the top label of $T_{2}$, and
\item there are wires $W_{i_{1}}, W_{i_{2}}, \ldots, W_{i_{n}} (n \geq 1)$ such that the bottom contacts $T_{2}$, read from left to right, are
precisely $$t(W_{i_{1}}), t(W_{i_{2}}), \ldots, t(W_{i_{n}})$$
and the top contacts of $T_{1}$, read from left to right, are precisely
$$ b(W_{i_{1}}), b(W_{i_{2}}), \ldots, b(W_{i_{n}}).$$
\end{enumerate}
Define a new braided diagram as follows. Remove the transistors $T_{1}$ and $T_{2}$ and all of the wires $W_{i_{1}}, \ldots, W_{i_{n}}$ connecting the top
of $T_{1}$ to the bottom of $T_{2}$. Let $W_{j_{1}}, \ldots, W_{j_{m}}$ be the wires attached (in that order) 
to the top of $T_{2}$, and let $W_{k_{1}}, \ldots, W_{k_{m}}$ be the wires attached to the bottom of $T_{1}$. We glue the bottom of
$W_{j_{\ell}}$ to the top of $W_{k_{\ell}}$. There is a natural well-defined labelling function on the resulting wires, since 
$\ell(W_{j_{\ell}}) = \ell(W_{k_{\ell}})$ by our assumptions. We say that the new diagram $\Delta'$ is obtained from $\Delta$ by \emph{reducing the dipole 
$(T_{1}, T_{2})$}. The inverse operation is called \emph{inserting a dipole}.
\end{definition}

\begin{definition} \label{def: equivalentdiagrams}
(Equivalent Diagrams)
We say that two diagrams $\Delta_{1}$, $\Delta_{2}$ are \emph{equivalent} if there is a homeomorphism 
$\phi: \Delta_{1} \rightarrow \Delta_{2}$ that preserves the labels on the wires, restricts to a homeomorphism
$\phi_{|}: F(\Delta_{1}) \rightarrow F(\Delta_{2})$, preserves the tops and bottoms of the transistors and frame, and preserves
the left to right orientations on the transistors and the frame. We write $\Delta_{1} \equiv \Delta_{2}$.
\end{definition}

\begin{center}
\begin{figure}[!h]
\includegraphics{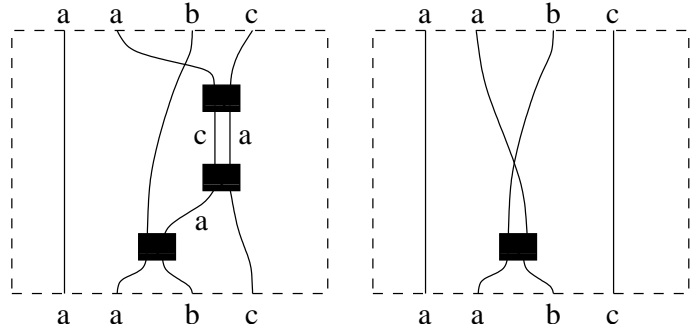}
\caption{The diagram on the right is obtained from the one on the left by reduction of a dipole.}
\label{figure2}
\end{figure}
\end{center}

\begin{definition} \label{def: equivalentmodulodipoles}
(Equivalent Modulo Dipoles; Reduced Diagram)
We say that $\Delta$ and $\Delta'$ are \emph{equivalent modulo dipoles}
if there is a sequence $\Delta \equiv \Delta_{1} \equiv \Delta_{2} \equiv \ldots \equiv \Delta_{n} \equiv \Delta'$,
where $\Delta_{i+1}$ is obtained from $\Delta_{i}$ by either inserting or removing a dipole, for $i \in \{ 1, \ldots, n-1 \}$. We write $\Delta = \Delta'$.
(The relation of equivalence modulo dipoles is indeed an equivalence relation -- see \cite{Far3}.)

A braided diagram $\Delta$ over a semigroup presentation is called \emph{reduced} if it contains no dipoles. Each equivalence class modulo dipoles
contains a unique reduced diagram \cite{Far3}.
\end{definition}

\begin{example}
In Figure \ref{figure2}, we have two braided diagrams over the semigroup presentation $\mathcal{P} = \langle a, b, c \mid ab=ba, ac=ca, bc=cb \rangle$.
The two rightmost transistors in the diagram on the left form a dipole, and the diagram on the right is the result of reducing that dipole. 
\end{example}

\begin{theorem} \label{thm: braideddiagramgroups} \cite{Far3}
Let $\mathcal{P} = \langle \Sigma \mid \mathcal{R} \rangle$ be a semigroup presentation, and let $w \in \Sigma^{+}$. We let $D_{b}(\mathcal{P},w)$ denote the set of equivalence classes of braided $(w,w)$-diagrams modulo dipoles. The operation of concatenation induces a well-defined group operation on $D_{b}(\mathcal{P},w)$. 
This group $D_{b}(\mathcal{P},w)$ is called \emph{the braided diagram group over $\mathcal{P}$ based at $w$}.
\end{theorem}

\section{Groups defined by finite similarity structures} \label{section:fss}

\subsection{Review of ultrametric spaces and finite similarity structures}
We now give a quick review of finite similarity structures on  compact ultrametric spaces, as defined in Hughes \cite{HugLSHP}. Most of this subsection is taken directly from 
\cite{FH}.

\begin{definition} \label{def:umetric} An {\it ultrametric space} is a metric space $(X,d)$ such that $d(x,y)\leq\max\{d(x,z), d(z,y)\}$ for all
$x,y,z\in X$.
\end{definition}

\begin{lemma} \label{lemma:ultrametric}
Let $X$ be an ultrametric space.
\begin{enumerate}
\item Let $B_{r}(x)$ be an open metric ball in $X$. If $y \in B_{r}(x)$, then $B_{r}(x) = B_{r}(y)$.
\item If $B_{1}$ and $B_{2}$ are open metric balls in $X$, then either the balls are disjoint, or one is contained in the other.
\item Every open ball in $X$ is a closed set, and every closed ball in $X$ is an open set.
\item If $X$ is compact, then each open ball $B$ is contained in at most finitely many distinct open balls of $X$.
\item If $X$ is compact, then each open ball in $X$ is a closed ball (possibly of a different radius), and each closed ball is an open ball. 
\item If $X$ is compact and $x$ is not an isolated point, then each open ball $B_{r}(x)$ is partitioned by its maximal proper open subballs, which are finite in number.
\end{enumerate}
\qed
\end{lemma}




\begin{convention}
We assume for the rest of the section that $X$ is a compact ultrametric space. By Lemma \ref{lemma:ultrametric}(5), open balls are closed balls, and closed balls are open balls,
so we can refer to both simply as balls, and we will follow this practice from now on.
\end{convention}

\begin{definition}
Let $f: X \rightarrow Y$ be a function between metric spaces. We say that $f$ is a \emph{similarity} if there is a constant $C>0$ such that
$d_{Y}(f(x_{1}), f(x_{2})) = Cd_{X}(x_{1},x_{2})$, for all $x_{1}$ and $x_{2}$ in $X$.
\end{definition}

\begin{definition}
\label{def:fin sim struct}
A
 {\it finite similarity structure for $X$} is a function
$\sm_X$ that assigns to
each ordered pair $B_1, B_2$ of balls in $X$
a (possibly empty)
set $\sm_X(B_1,B_2)$ of surjective similarities
$B_1\to B_2$ such that whenever 
$B_1, B_2, B_3$ are balls in $X$, the following properties
hold:
\begin{enumerate}
\item (Finiteness) $\sm_X(B_1,B_2)$ is a finite set.
\item (Identities) $\id_{B_1}\in\sm_X(B_1,B_1)$.
\item (Inverses) If $h\in\sm_X(B_1,B_2)$, then $h^{-1}\in\sm_X(B_2,B_1)$.
\item (Compositions) 
If $h_1\in\sm_X(B_1,B_2)$ and $h_2\in\sm_X(B_2,B_3)$, then
$h_2h_1\in\sm_X(B_1,B_3)$.
\item (Restrictions)
If $h\in\sm_X(B_1,B_2)$ and $B_3\subseteq B_1$, 
then $$h|B_3\in\sm_X(B_3,h(B_3)).$$
\end{enumerate}
\end{definition}

\begin{definition}
A homeomorphism $h\co X\to X$ is \emph{locally determined by $\sm_{X}$} provided that for every $x \in X$, there exists a ball
$B'$ in $X$ such that $x \in B'$, $h(B')$ is a ball in $X$, and $h|B'\in\sm(B', h(B'))$.
\end{definition}


\begin{definition}
The \emph{finite similarity structure} 
(\emph{FSS}) \emph{group} 
$\Gamma(\sm_X)$ is the set of all homeomorphisms $h\co X\to X$ such that $h$ is locally determined by
$\sm_X$.
\end{definition}

\begin{remark}
The fact that $\Gamma(\sm_{X})$ is a group under composition is due to Hughes \cite{HugLSHP}.
\end{remark}

\subsection{A description of the homeomorphisms determined by a similarity structure} In this subsection, we offer a somewhat simpler description
of the elements in the groups $\Gamma(\sm_{X})$ (Proposition \ref{prop:simpledescription}), which shows that elements $\gamma \in \Gamma(\sm_{X})$ can be described in a manner reminiscent of the tree pair representatives for elements in Thompson's group $V$ (see \cite{CFP}). 

\begin{definition}
We define the \emph{standard partitions} of $X$ inductively as follows.
\begin{enumerate}
\item $\{ X \}$ is a standard partition. 
\item If $\mathcal{P} = \{ \widehat{B}_{1}, \ldots, \widehat{B}_{n} \}$ is a standard partition, and $\{ B_{1}, \ldots, B_{m} \}$ is the partition of $\widehat{B}_{i}$ into maximal proper subballs, then
$(\mathcal{P} - \{ \widehat{B}_{i} \}) \cup \{ B_{1}, \ldots, B_{m} \}$ is also a standard partition.
\end{enumerate}
Clearly, each standard partition is a partition of $X$ into balls.
\end{definition}

\begin{lemma} \label{lemma:standard}
Every partition $\mathcal{P}$ of $X$ into balls is standard. 
\end{lemma}

\begin{proof}
We prove this by induction on $|\mathcal{P}|$. It is clearly true if $|\mathcal{P}| = 1$. We note that compactness implies that each 
partition $\mathcal{P}$ of $X$ into balls must be finite.

For an arbitrary ball $B \subseteq X$, we define the \emph{depth} of $B$, denoted $d(B)$, to be the number of distinct balls
of $X$ that contain $B$. (This definition is similar to Definition 3.19 from \cite{FH}.) We note that $d(B)$ is a positive integer by Lemma \ref{lemma:ultrametric}(4), and $d(X) = 1$.

Now we suppose that a partition $\mathcal{P}$ is given to us. We assume inductively that all partitions with smaller numbers of
balls are standard. By finiteness of $\mathcal{P}$, 
there is some ball $B$ having maximum depth $m$, where we can assume that $m \geq 2$. Let $\widehat{B}$ denote the ball
containing $B$ as a maximal proper subball. Clearly, $d(\widehat{B}) = m-1$. We let $\{ B_{0}, \ldots, B_{k} \}$ be the collection
of maximal proper subballs of $\widehat{B}$, where $B = B_{0}$ and $k \geq 1$.  

We claim that $\{ B_{0}, B_{1}, B_{2}, \ldots, B_{k} \} \subseteq \mathcal{P}$. Choose $x \in B_{i}$. Our assumptions
imply that $x$ is in some ball $B'$ of $\mathcal{P}$ such that $d(B') \leq m$. The only such balls are $B_{i}$, $\widehat{B}$, and
any balls that contain $\widehat{B}$. (This uses an appeal to Lemma \ref{lemma:ultrametric}(2).) Since $\widehat{B} \cap B_{0} \neq \emptyset$ and $B_{0} = B \in \mathcal{P}$,
the only possibility is that $B' = B_{i}$, since $\mathcal{P}$ is a partition. This proves that $\{ B_{0}, \ldots, B_{k} \} \subseteq \mathcal{P}$. 

Now we consider the partition $\mathcal{P}' = (\mathcal{P} - \{ B_{0}, \ldots, B_{k} \}) \cup \{ \widehat{B} \}$. 
This partition is standard by the induction hypothesis, and it follows directly that $\mathcal{P}$ itself is standard. 
\end{proof}

\begin{proposition} \label{prop:simpledescription}
Let $\sm_{X}$ be a finite similarity structure on $X$, and let $\gamma \in \Gamma(\sm_{X})$. There exist standard partitions $\mathcal{P}_{1} = \{ B_{1}, \ldots, B_{n} \}$ and $\mathcal{P}_{2}$ of $X$, a bijection $\phi: \mathcal{P}_{1} \rightarrow \mathcal{P}_{2}$, and elements $\sigma_{i} \in \sm(B_{i}, \phi(B_{i}))$
such that $\gamma_{\mid B_{i}} = \sigma_{i}$, for $i = 1, \ldots, n$. 

Moreover, we can arrange that the balls $B_{i}$ are maximal in the sense that if $B \subseteq X$ and $\gamma(B)$ are balls such 
that $\gamma_{\mid B} \in \sm_{X}(B, \gamma(B))$, then $B \subseteq B_{i}$, for some $i \in \{ 1, \ldots, n \}$.
\end{proposition}

\begin{proof}
Since $\gamma$ is locally determined by $\sm_{X}$, we can find an open cover of $X$ by balls such that the restriction of $\gamma$ to each ball is a local similarity in the
$\sm_{X}$-structure. By compactness of $X$, we can pass to a finite subcover. An application of Lemma \ref{lemma:ultrametric}(2) allows us to pass to a subcover that is also 
a partition. We call this partition $\mathcal{P}_{1}$. We can then set $\mathcal{P}_{2} = \gamma( \mathcal{P}_{1})$. Both partitions are standard by Lemma \ref{lemma:standard}.

The final statement is essentially Lemma 3.7 from \cite{HugLSHP}.
\end{proof}

\section{Braided diagram groups and groups determined by finite similarity structures} \label{section:main}

\subsection{Braided diagram groups over tree-like semigroup presentations}
\begin{definition} \label{def: treelike}
A semigroup presentation $\mathcal{P} = \langle \Sigma \mid \mathcal{R} \rangle$ is \emph{tree-like} if,
\begin{enumerate}
\item every relation $(w_{1}, w_{2}) \in \mathcal{R}$ satisfies $|w_{1}| = 1$ and $|w_{2}| > 1$;
\item if $(a, w_{1}), (a, w_{2}) \in \mathcal{R}$, then $w_{1} \equiv w_{2}$.
\end{enumerate}
\end{definition}

\begin{example} \label{example:Thompson}
The generalized Thompson's groups $V_{d}$ are isomorphic to the braided diagram groups $D_{b}(\mathcal{P}, x)$, 
where $\mathcal{P} = \langle x \mid (x,x^{d}) \rangle$ is a tree-like semigroup presentation.
This fact was already proved in \cite{GubaSapir} and \cite{Far3}, and it is also a consequence of Theorem \ref{thm: bigisomorphism}. 
\end{example}  

\begin{example} \label{example:houghton}
Consider the graph $G_{n}$ made up from a disjoint union of $n$ rays:  $G_{n} = \{ 1, \ldots, n \} \times [0,\infty)$. 
We assume that each ray is given the standard CW-complex structure with a vertex at
each integer. We define the \emph{Houghton group} $H_{n}$ to be the set of bijections $h$ of the vertices $G_{n}^{0}$ such that
\begin{enumerate}
\item $h$ preserves adjacency, with at most finitely many exceptions, and
\item $h$ preserves ends: that is, for each $i \in \{ 1, \ldots, n \}$, there are $i_{1}, i_{2} \in \mathbb{N}$ such that $h( \{ i \} \times [i_{1}, \infty)) = \{ i \} \times [i_{2}, \infty)$.
\end{enumerate}
Ken Brown \cite{Brown1987} showed that $H_{n}$ is a group of type $F_{n-1}$ but not of type $F_{n}$.

We will sketch a proof that each Houghton group $H_{n}$ is a braided diagram group over a tree-like semigroup presentation. (It follows, in particular, that each of these
groups is $\Gamma(\sm_{X})$, for an appropriate compact ultrametric space $X$ and finite similarity structure $\sm_{X}$, by Theorem \ref{thm: bigisomorphism}.)
For $n \geq 2$, consider the semigroup presentation 
$$ \mathcal{P}_{n} = \langle a, r, x_{1}, \ldots, x_{n} \mid (r, x_{1}x_{2}x_{3} \ldots x_{n}), (x_{1}, ax_{1}), (x_{2}, ax_{2}), \ldots, (x_{n}, ax_{n})  \rangle.$$
Similarly, we can define $\mathcal{P}_{1} = \langle a, r \mid (r, ar) \rangle$.
We claim that $D_{b}(\mathcal{P}_{n}, r)$ is isomorphic to $H_{n}$. We sketch the proof for $n \geq 2$; the proof in case $n=1$ is very similar. 

The elements of $D_{b}(\mathcal{P}_{n}, r)$ can be expressed in the form $\Delta_{2} \circ \Delta_{1}^{-1}$, where each transistor in $\Delta_{i}$ is ``positive"; i.e., the top label of the transistor is the left side of a relation in $\mathcal{P}_{n}$, and the bottom label is the right side. (This is proved as part of the proof of Theorem \ref{thm: bigisomorphism}.)
We can think of each diagram $\Delta_{i}$ as a recipe for separating $G_{n}$ into connected components. The wires running between $\Delta_{1}$ and $\Delta_{2}$
in the concatenation $\Delta_{2} \circ \Delta_{1}^{-1}$ describe how these connected components should be matched by the bijection $h \in H_{n}$. To put it more explicitly, 
the relations represent the following operations:
\begin{enumerate}
\item the relation $(r, x_{1}x_{2} \ldots x_{n})$ describes the initial configuration $G_{n}$ of $n$ disjoint rays. The letters $x_{i}$ ($i \in \{ 1, \ldots, n \}$) represent the isomorphism types of the different rays.
The different subscripts prevent different ends of $G_{n}$ from being permuted nontrivially. (If we wish to remove the end-preserving condition above, we can simply replace the
$n$ distinct symbols $x_{1}$, $x_{2}$, $\ldots$, $x_{n}$ by the single symbol $x$.)
\item the relation $(x_{i}, ax_{i})$ (for $i \in \{ 1, \ldots, n \}$) represents the action of breaking the initial vertex away from the ray of isomorphism type $x_{i}$. The initial vertex of the ray gets the label $a$, and the new ray retains the label $x_{i}$ (since it is of the same combinatorial type as the original ray, and the new ray is a permissible target for the original ray
under the action of the Houghton group). The letter $a$ thus represents a single floating vertex. Any two such vertices can be matched by an element of the Houghton group, which 
is why we use a single label for all of these vertices.
\end{enumerate}

We illustrate how an $(r,r)$-diagram over $\mathcal{P}_{2}$ represents an element of $H_{2}$. Figure \ref{figure3} depicts an $(r,r)$-diagram $\Delta_{h}$ over $\mathcal{P}_{2}$.
This diagram $\Delta_{h}$ represents the element $h \in H_{2}$ that sends: $(2,n)$ to $(2,n-1)$, for each $n \geq 1$, $(1,n)$ to $(1,n+1)$ (for all $n$), and $(2,0)$ to $(1,0)$.
Note that the bottom portion of the diagram represents a subdivision of the domain, and the top portion represents a subdivision of the range.
It is straightforward to check that the indicated function does not change if we insert or remove dipoles.
\begin{center}
\begin{figure}[!h]
\includegraphics{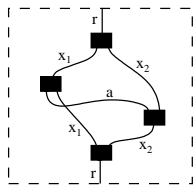}
\caption{The given $(r,r)$-diagram over $\mathcal{P}_{2}$ represents an element of Houghton's group $H_{2}$.}
\label{figure3}
\end{figure}
\end{center}
\end{example}

\begin{example} \label{example:Q}
The same principle can be used to exhibit the group $\Gamma = \mathrm{QAut}(\mathcal{T}_{2,c})$ as a braided diagram group over a tree-like
semigroup presentation. Here $\Gamma$ (as defined in \cite{BMN}) is the group of self-bijections $h$ of the vertices of the infinite ordered rooted binary tree $T$ such that
\begin{enumerate}
\item $h$ preserves adjacency, with at most finitely many exceptions, and
\item $h$ preserves the left-right ordering of the edges incident with and below a given vertex, again with at most finitely many exceptions.
\end{enumerate}
Consider the semigroup presentation $\mathcal{P} = \langle a, x \mid (x, xax) \rangle$. We claim that $D_{b}(\mathcal{P}, x)$ is isomorphic to $\Gamma$. Much of the discussion from the previous example carries over identically.
We will simply indicate how the single relation allows us to simulate breaking the binary tree into pieces. (Such a dissection of $T$ would be represented by
a positive diagram, as above. The wires connecting the bottoms of the positive diagrams $\Delta_{1}$ and $\Delta_{2}$ would again represent how the resulting pieces are matched by
a bijection.)

The letter $x$ represents the isomorphism type of the binary tree $T$. The relation represents breaking the tree at the root. The result of this operation yields a floating vertex
(represented by the letter $a$), and two new rooted binary trees (both represented by $x$). The first $x$ in $xax$ represents the left branch, and the second $x$ represents the
right. The description of the isomorphism of $D_{b}(\mathcal{P}, x)$ with $\Gamma$ now follows the general pattern of the previous example. 
\end{example}

\begin{remark} 
Example \ref{example:houghton} shows that the $F_{\infty}$ result of \cite{FH} cannot be extended to all groups determined by finite similarity structures (as defined in Section \ref{section:fss}). 

All of the groups in the above examples act properly on CAT(0) cubical complexes by a construction of \cite{Far3}.

We note also that the representation of the above groups as braided diagram groups suggests a method for producing embeddings into other groups, such as (perhaps most notably) Thompson's
group $V$. For instance, the group from Example \ref{example:Q} can be embedded into $V$ as follows. Given a braided diagram $\Delta$ over $\mathcal{P}$, systematically
replace each $a$ label with an $x$. The result is a braided diagram over the semigroup presentation $\mathcal{P}' = \langle x \mid (x,x^{3}) \rangle$. The indicated function 
$\phi: D_{b}(\mathcal{P}, x) \rightarrow D_{b}(\mathcal{P}', x)$ is
easily seen to be a homomorphism, and $\phi$ is injective since it sends reduced diagrams to reduced diagrams. We can now appeal to the fact that
$D_{b}(\mathcal{P}', x) \cong V_{3}$ (the $3$-ary version of Thompson's group $V$), and the latter group embeds in $V$ itself. 

The group from Example \ref{example:Q} was previously known to embed in $V$ by a result of \cite{BMN}.
\end{remark}

\subsection{Groups determined by small similarity structures}

\begin{definition} \label{def: smallsim} 
Let $X$ be a compact ultrametric space. We say that the finite similarity structure $\sm_{X}$ is \emph{small}
if, for every pair of balls $B_{1}, B_{2}$ in $X$, $|\sm_{X}(B_{1},B_{2})| \leq 1$.
\end{definition}

\begin{definition} \label{def:localorder}
Let $X$ be a compact ultrametric space endowed with a small similarity structure $\sm_{X}$. If $B \subseteq X$ is a ball in
$X$ that is not an isolated point, then a \emph{local ball order} at $B$ is an assignment of a linear order $<$ to the set $\{ \widehat{B}_{1}, \ldots, \widehat{B}_{n} \}$
of maximal proper subballs of $B$. A \emph{ball order} on $X$ is an assignment of such a linear order to each ball $B \subseteq X$ that is not a singleton. The ball order is 
\emph{compatible} with $\sm_{X}$ if each $h \in \sm_{X}(B_{1}, B_{2})$ induces an order-preserving bijection of the maximal proper subballs of $B_{1}$ and $B_{2}$,
for all choices of $B_{1}$ and $B_{2}$.
\end{definition}

\begin{lemma} \label{lemma:compatibleballorder}
Let $X$ be a compact ultrametric space endowed with a small similarity structure. There exists a ball order on $X$ that is compatible with $\sm_{X}$.
\end{lemma}

\begin{proof}
We recall a definition from \cite{FH}. 
Let $B \subseteq X$ be a metric ball. Let
$[B] = \{ B' \subseteq X \mid \sm_{X}(B, B') \neq \emptyset \}$; $[B]$ is called the \emph{$\sm_{X}$-class} of $B$.

From a given $\sm_{X}$-class of balls, choose a ball $B$. If $B$ is not a singleton, then there exists a collection of maximal proper subballs
$B_{1}$, $\ldots$, $B_{n}$ of $B$. Choose a linear order on this collection of balls; without loss of generality, $B_{1} < B_{2} < \ldots < B_{n}$. If $B'$ is another ball
in $[B]$, then we can let $h$ denote the unique element of $\sm_{X}(B, B')$. This $h$ carries the maximal proper subballs of $B$ into maximal proper subballs of $B'$, and thereby 
induces an order $h(B_{1}) < h(B_{2}) < \ldots < h(B_{n})$ on the maximal proper subballs of $B'$. This procedure gives a local ball order to each ball $B' \in [B]$.

We repeat this procedure for each $\sm_{X}$-class of balls. The result is a ball order on $X$ that is compatible with $\sm_{X}$.
\end{proof}

\begin{remark} \label{remark:ballorder}
A ball order on $X$ also determines a linear order on any given collection of pairwise disjoint balls in $X$. For let $\mathcal{C}$
be such a collection, and let $B_{1}, B_{2} \in \mathcal{C}$. There is a unique smallest ball $B \subseteq X$ that contains
both $B_{1}$ and $B_{2}$, by Lemma \ref{lemma:ultrametric}(4). Let $\{ \widehat{B}_{1}, \ldots, \widehat{B}_{n} \}$
be the collection of maximal proper subballs of $B$. By minimality of $B$, 
we must have that $B_{1}$ and $B_{2}$ are contained in distinct 
maximal proper subballs of $B$; say $B_{1} \subseteq \widehat{B}_{1}$ and $B_{2} \subseteq \widehat{B}_{2}$. We write
$B_{1} < B_{2}$ if and only if $\widehat{B}_{1} < \widehat{B}_{2}$. This defines a linear order on $\mathcal{C}$. 
The verification is straightforward.
\end{remark}

\begin{definition} \label{def: psimx}
Let $X$ be a compact ultrametric space with a small similarity structure $\sm_{X}$ and a compatible ball order.
Define a semigroup presentation $\mathcal{P}_{\sm_{X}} = \langle \Sigma \mid \mathcal{R} \rangle$
as follows. Let 
$$ \Sigma = \{ [B] \mid B \text{ is a ball in } X \}.$$
If $B \subseteq X$ is a ball, let $B_{1}, \ldots, B_{n}$ be the maximal proper subballs
of $B$, listed in order. If $B$ is a point, then $n=0$. We set
$$ \mathcal{R} = \{ ( [B], [B_{1}][B_{2}]\ldots[B_{n}]) \mid n \geq 1, B \text{ is a ball in } X \}.$$
\end{definition}

\begin{remark} \label{rem: psimx}
We note that $\mathcal{P}_{\sm_{X}}$ will always be a tree-like semigroup presentation, for any choice of
compact ultrametric space $X$, small similarity structure $\sm_{X}$, and compatible ball order.
\end{remark}

\subsection{The main theorem}

\begin{theorem} \label{thm: bigisomorphism}
If $X$ is a compact ultrametric space with a small similarity structure $\sm_{X}$ and compatible ball order, 
then 
$$ \Gamma(\sm_{X}) \cong D_{b}(\mathcal{P}_{\sm_{X}}, [X]).$$
Conversely, if $\mathcal{P} = \langle \Sigma \mid \mathcal{R} \rangle$ is a tree-like semigroup presentation,
and $x \in \Sigma$, then there is a compact ultrametric space $X_{\mathcal{P}}$, a small
finite similarity structure $\sm_{X_{\mathcal{P}}}$, and a compatible ball order such that
$$D_{b}(\mathcal{P}, x) \cong \Gamma(\sm_{X_{\mathcal{P}}}).$$  
\end{theorem}

\begin{proof}
If $\gamma \in \Gamma(\sm_{X})$, then, by Proposition \ref{prop:simpledescription}, there are standard 
partitions $\mathcal{P}_{1}$, $\mathcal{P}_{2}$ of $X$ into balls,
and a bijection $\phi: \mathcal{P}_{1} \rightarrow \mathcal{P}_{2}$ such that, for any $B \in \mathcal{P}_{1}$,
$\gamma(B) = \phi(B)$ and $\gamma|_{B} \in \sm_{X}(B, \gamma(B))$.
Since $|\sm_{X}(B,\gamma(B))| \leq 1$, the triple $(\mathcal{P}_{1}, \mathcal{P}_{2}, \phi)$ determines $\gamma$ without
ambiguity. We call $(\mathcal{P}_{1}, \mathcal{P}_{2}, \phi)$ a \emph{defining triple} for $\gamma$. Note that a given $\gamma$
will usually have many defining triples. Let $\mathcal{D}$ be the set of all defining triples, for $\gamma$ running over all of
$\Gamma(\sm_{X})$.

We will now define a map $\psi: \mathcal{D} \rightarrow D_{b}(\mathcal{P}_{\sm_{X}}, [X])$. To a partition $\mathcal{P}$
of $X$ into balls, we first assign a braided diagram $\Delta_{\mathcal{P}}$ over $\mathcal{P}_{\sm_{X}}$. There is a transistor 
$T_{B} \in \mathcal{T}(\Delta_{\mathcal{P}})$ for each ball $B$ which properly contains some ball of $\mathcal{P}$. There is a wire
$W_{B} \in \mathcal{W}(\Delta_{\mathcal{P}})$ for each ball $B$ which contains a ball of $\mathcal{P}$. The wires are attached as follows:
\begin{enumerate}
\item If $B=X$, then we attach the top of $W_{B}$ to the top of the frame.
 If $B \neq X$, then the top of the wire $W_{B}$ is attached to the bottom of the transistor $T_{\widehat{B}}$, where $\widehat{B}$ is the (unique) ball
that contains $B$ as a maximal proper subball.

Moreover, we attach the wires in an ``order-respecting" fashion. Thus, if $\widehat{B}$ is a ball properly containing balls of $\mathcal{P}$, we let 
$B_{1}, B_{2}, \ldots, B_{n}$ be the collection of maximal proper subballs of $\widehat{B}$, listed in order. We attach the wires $W_{B_{1}}, W_{B_{2}}, \ldots,
W_{B_{n}}$ so that $t(W_{B_{i}})$ is to the left of $t(W_{B_{j}})$ on the bottom of $T_{\widehat{B}}$ if $i < j$.

\item The bottom of the wire $W_{B}$ is attached to the top of $T_{B}$ if $B$ properly contains a ball of $\mathcal{P}$. If not (i.e., if $B \in \mathcal{P}$),
then we attach the bottom of $W_{B}$ to the bottom of the frame. We can arrange, moreover, that the wires are attached in an order-respecting manner to the
bottom of the frame. (Thus, if $B_{1} < B_{2}$ ($B_{1}, B_{2} \in \mathcal{P}$), we have that $b(W_{B_{1}})$ is to the left of $b(W_{B_{2}})$.)
\end{enumerate}
The labelling function $\ell: \mathcal{W}(\Delta_{\mathcal{P}}) \rightarrow \Sigma$ sends $W_{B}$ to $[B]$. It is straightforward to check that the
resulting $\Delta_{\mathcal{P}}$ is a braided diagram over $\mathcal{P}_{\sm_{X}}$. The top label of $\Delta_{\mathcal{P}}$ is $[X]$.

Given a bijection $\phi: \mathcal{P}_{1} \rightarrow \mathcal{P}_{2}$, where $\mathcal{P}_{1}$ and $\mathcal{P}_{2}$ are partitions of $X$ into balls
and $[B] = [\phi(B)]$, we can define a braided diagram $\Delta_{\phi}$ over $\mathcal{P}_{\sm_{X}}$ as follows. We let 
$\mathcal{T}(\Delta_{\phi}) = \emptyset$, and $\mathcal{W}(\Delta_{\phi}) = \{ W_{B} \mid B \in \mathcal{P}_{1} \}$.  
We attach the top of each wire to the frame in such a way that $t(W_{B_{1}})$ is to the left of $t(W_{B_{2}})$ if $B_{1} < B_{2}$.
(Here $<$ refers to the ordering from Remark \ref{remark:ballorder}.) We attach the
bottom of each wire to the bottom of the frame in such a way that $b(W_{B_{1}})$ is to the left of $b(W_{B_{2}})$ if $\phi(B_{1}) < \phi(B_{2})$.

Now, for a defining triple $(\mathcal{P}_{1}, \mathcal{P}_{2}, \phi) \in \mathcal{D}$, we set
$\psi( (\mathcal{P}_{1}, \mathcal{P}_{2}, \phi)) = \Delta_{\mathcal{P}_{2}} \circ \Delta_{\phi^{-1}} \circ \Delta_{\mathcal{P}_{1}}^{-1}
\in \mathcal{D}_{b}(\mathcal{P}_{\sm_{X}}, [X])$.

We claim that any two defining triples $(\mathcal{P}_{1}, \mathcal{P}_{2}, \phi)$, $(\mathcal{P}_{1}', \mathcal{P}_{2}', \phi')$ 
for a given $\gamma \in \Gamma(\sm_{X})$
have the same image
in $D_{b}(\mathcal{P}_{\sm_{X}}, [X])$, modulo dipoles. We begin by proving an intermediate statement. Let $(\mathcal{P}_{1}, \mathcal{P}_{2}, \phi)$
be a defining triple. Let $B \in \mathcal{P}_{1}$, and let $\widehat{B}_{1}, \ldots, \widehat{B}_{n}$ be the collection of maximal proper subballs of $B$, listed
in order. We let $B' = \phi(B)$ and let $\widehat{B}'_{1}, \ldots, \widehat{B}'_{n}$ be the collection of maximal proper subballs of $B'$. (Note that
$[B'] = [B]$ by our assumptions, so both have the same number of maximal proper subballs, and in fact $[\widehat{B}_{i}] = [\widehat{B}'_{i}]$ for $i = 1, \ldots, n$,
since $\gamma|_{B} \in \sm_{X}(B, B')$ and the elements of $\sm_{X}(B, B')$ preserve order.) We set 
$\widehat{\mathcal{P}}_{1} = (\mathcal{P}_{1} - \{ B \}) \cup \{ \widehat{B}_{1}, \ldots, \widehat{B}_{n} \}$,
$\widehat{\mathcal{P}}_{2} = (\mathcal{P}_{2} - \{ B' \}) \cup \{ \widehat{B}'_{1}, \ldots, \widehat{B}'_{n} \}$,
and $\widehat{\phi}|_{\mathcal{P}_{1} - \{ B \}} = \phi|_{\mathcal{P}_{1} - \{ B \}}$, 
$\widehat{\phi} (\widehat{B}_{i}) = \widehat{B}'_{i}$. We say that $(\widehat{\mathcal{P}}_{1}, \widehat{\mathcal{P}}_{2}, \widehat{\phi})$
is obtained from  $(\mathcal{P}_{1}, \mathcal{P}_{2}, \phi)$ by \emph{subdivision} at $(B, B')$.
A straightforward argument shows that $\psi( (\widehat{\mathcal{P}}_{1}, \widehat{\mathcal{P}}_{2}, \widehat{\phi}))$ is in fact obtained from
$\psi( (\mathcal{P}_{1}, \mathcal{P}_{2}, \phi))$ by inserting a dipole. We omit the details, which rely on the fact that each
element of the $\sm_{X}$-structure preserves the local ball order.

Now suppose that $(\mathcal{P}_{1}, \mathcal{P}_{2}, \phi)$ and $(\mathcal{P}'_{1}, \mathcal{P}'_{2}, \phi')$ are defining triples
for the same element $\gamma \in \Gamma(\sm_{X})$. We can find a common refinement $\mathcal{P}''_{1}$ of $\mathcal{P}_{1}$ and $\mathcal{P}'_{1}$.
Using the fact that all partitions of $X$ into balls are standard (Lemma
\ref{lemma:standard}),
 we can pass from $(\mathcal{P}_{1}, \mathcal{P}_{2}, \phi)$ to $(\mathcal{P}''_{1}, \widehat{P}_{2}, \widehat{\phi})$ by repeated subdivision
(for some partition $\widehat{\mathcal{P}}_{2}$ of $X$ into balls and some bijection $\widehat{\phi}: \mathcal{P}''_{1} \rightarrow \widehat{\mathcal{P}}_{2}$).
Since subdivision does not change the values of $\psi$ modulo dipoles, $\psi((\mathcal{P}_{1}, \mathcal{P}_{2}, \phi)) = 
\psi((\mathcal{P}''_{1}, \widehat{\mathcal{P}}_{2}, \widehat{\phi}))$ modulo dipoles. Similarly, we can subdivide
$(\mathcal{P}'_{1}, \mathcal{P}'_{2}, \phi')$ repeatedly in order to obtain $(\mathcal{P}''_{1}, \widehat{\mathcal{P}}'_{2}, \widehat{\phi}')$, where
$\psi((\mathcal{P}'_{1}, \mathcal{P}'_{2}, \phi')) = \psi((\mathcal{P}''_{1}, \widehat{\mathcal{P}}'_{2}, \widehat{\phi}'))$ modulo dipoles.
Both $(\mathcal{P}''_{1}, \widehat{\mathcal{P}}'_{2}, \widehat{\phi}')$
and $(\mathcal{P}''_{1}, \widehat{\mathcal{P}}_{2}, \widehat{\phi})$ are defining triples for $\gamma$, so we are forced to have
$\widehat{\phi} = \widehat{\phi}'$ and $\widehat{\mathcal{P}}_{2} = \widehat{\mathcal{P}}'_{2}$.
It follows that $\psi((\mathcal{P}_{1}, \mathcal{P}_{2}, \phi)) = \psi((\mathcal{P}'_{1}, \mathcal{P}'_{2}, \phi'))$, so $\psi$ induces a function
from $\Gamma(\sm_{X})$ to $D_{b}(\mathcal{P}_{\sm_{X}}, [X])$. We will call this function $\widehat{\psi}$.

Now we will show that $\widehat{\psi}: \Gamma(\sm_{X}) \rightarrow D_{b}(\mathcal{P}_{\sm_{X}}, [X])$ is a homomorphism. Let $\gamma, \gamma' \in
\Gamma(\sm_{X})$. After subdividing as necessary, we 
can choose defining triples $(\mathcal{P}_{1}, \mathcal{P}_{2}, \phi)$ and $(\mathcal{P}'_{1}, \mathcal{P}'_{2}, \phi')$ for 
$\gamma$ and $\gamma'$ (respectively) in such a way that $\mathcal{P}_{2} = \mathcal{P}'_{1}$. It follows easily that 
$(\mathcal{P}_{1}, \mathcal{P}'_{2}, \phi' \phi)$ is a defining triple for $\gamma' \gamma$. Therefore, 
$\widehat{\psi}(\gamma' \gamma) = \Delta_{\mathcal{P}'_{2}} \circ \Delta_{(\phi' \phi)^{-1}} \circ \Delta^{-1}_{\mathcal{P}_{1}}$.
Now 
\begin{align*}
\widehat{\psi}(\gamma') \circ \widehat{\psi}(\gamma) &= \Delta_{\mathcal{P}'_{2}} \circ \Delta_{(\phi')^{-1}} \circ \Delta^{-1}_{\mathcal{P}'_{1}} \circ
\Delta_{\mathcal{P}_{2}} \circ \Delta_{\phi^{-1}} \circ \Delta^{-1}_{\mathcal{P}_{1}} \\
&= \Delta_{\mathcal{P}'_{2}} \circ \Delta_{(\phi')^{-1}} \circ \Delta_{\phi^{-1}} \circ \Delta^{-1}_{\mathcal{P}_{1}} \\
&= \Delta_{\mathcal{P}'_{2}} \circ \Delta_{(\phi'\phi)^{-1}} \circ \Delta^{-1}_{\mathcal{P}_{1}}
\end{align*}
Therefore, $\widehat{\psi}$ is a homomorphism. 
   
We now show that $\widehat{\psi}: \Gamma(\sm_{X}) \rightarrow D_{b}(\mathcal{P}_{\sm_{X}}, [X])$ is injective. Suppose that $\widehat{\psi}(\gamma) = 1$.
Using the final statement of Proposition \ref{prop:simpledescription}, we choose a defining triple $(\mathcal{P}_{1}, \mathcal{P}_{2}, \phi)$ for $\gamma$ with the property that, if $B \subseteq X$ is a ball, $\gamma(B)$ is a ball,
and $\gamma|_{B} \in \sm_{X}(B, \gamma(B))$, then $B$ is contained in some ball of $\mathcal{P}_{1}$. We claim that 
$\psi((\mathcal{P}_{1}, \mathcal{P}_{2}, \phi))$ is a reduced diagram. If there were a dipole $(T_{1}, T_{2})$, then we would have
$T_{1} \in \mathcal{T}(\Delta_{\mathcal{P}_{1}}^{-1})$ and $T_{2} \in \mathcal{T}(\Delta_{\mathcal{P}_{2}})$, since it is impossible for $\Delta_{\mathcal{P}}$
to contain any dipoles, for any partition $\mathcal{P}$ of $X$ into balls. Thus $T_{1} = T_{B_{1}}$ and $T_{2} = T_{B_{2}}$, where $[B_{1}] = [B_{2}]$ and
the wires from the bottom of $T_{B_{2}}$ attach to the top of $T_{B_{1}}$, in order.
This means that, if $\widehat{B}_{1}, \ldots, \widehat{B}_{n}$ are the maximal proper subballs of $B_{1}$, 
and $\widehat{B}'_{1}, \ldots, \widehat{B}'_{n}$ are the maximal proper subballs of $B_{2}$, then $\gamma(\widehat{B}_{i}) = \widehat{B}'_{i}$, where the latter
is a ball, and $\gamma|_{\widehat{B}_{i}} \in \sm_{X}(\widehat{B}_{i}, \widehat{B}'_{i})$.

Now, since $[B_{1}] = [B_{2}]$, there is $h \in \sm_{X}(B_{1}, B_{2})$. Since $\sm_{X}$ is closed under restrictions and $h$ preserves order, 
we have $h_{i} \in \sm_{X}( \widehat{B}_{i}, \widehat{B}'_{i})$ for $i = 1, \ldots, n$, where $h_{i} = h|_{\widehat{B}_{i}}$. It follows that
$\gamma|_{\widehat{B}_{i}} = h_{i}$, so, in particular, $\gamma|_{B_{1}} = h$.
Since $B_{1}$ properly contains some ball in $\mathcal{P}_{1}$, this is a contradiction. Thus, $\psi((\mathcal{P}_{1}, \mathcal{P}_{2}, \phi))$
is reduced.

We claim that $\psi((\mathcal{P}_{1}, \mathcal{P}_{2}, \phi))$ contains no transistors (due to the condition $\widehat{\psi}(\gamma) = 1$).
We have shown that $\psi((\mathcal{P}_{1}, \mathcal{P}_{2}, \phi))$ is a reduced diagram in the same class as the identity
$1 \in D_{b}(\mathcal{P}_{\sm_{X}}, [X])$. The identity can be represented as the (unique) $([X], [X])$-diagram $\Delta_{1}$ with only a single wire,
$W_{X}$, and no transistors. We must have $\psi((\mathcal{P}_{1}, \mathcal{P}_{2}, \phi)) \equiv \Delta_{1}$. Thus,
there is no ball that properly contains a ball of $\mathcal{P}_{1}$. It can only be that $\mathcal{P}_{1} = \{ X \}$, so we must have
$\gamma \in \sm_{X}(X,X)$. This forces $\gamma = 1$, so $\widehat{\psi}$ is injective.

Finally we must show that $\widehat{\psi}: \Gamma(\sm_{X}) \rightarrow D_{b}(\mathcal{P}_{\sm_{X}}, [X])$ is surjective. Let
$\Delta$ be a reduced $([X],[X])$-diagram over $\mathcal{P}_{\sm_{X}}$. A transistor $T \in \mathcal{T}(\Delta)$ is called \emph{positive}
if its top label is the left side of a relation in $\mathcal{P}_{\sm_{X}}$, otherwise (i.e., if the top label is the right side of a relation
in $\mathcal{P}_{\sm_{X}}$) the transistor $T$ is \emph{negative}. It is easy to see that the sets of positive and negative transistors partition
$\mathcal{T}(\Delta)$. We claim that, if $\Delta$ is reduced, then we cannot have $T_{1} \dot{\preceq} T_{2}$ when $T_{1}$ is positive and $T_{2}$
is negative. If we had such $T_{1} \dot{\preceq} T_{2}$, then we could find $T'_{1} \preceq T'_{2}$, where $T'_{1}$ is positive and $T'_{2}$ is negative.
Since $T'_{1}$ is positive, there is only one wire $W$ attached to the top of $T'_{1}$. This wire must be attached to the bottom of $T'_{2}$, since
$T'_{1} \preceq T'_{2}$, and it must be the only wire attached to the bottom of $T'_{2}$, since $T'_{2}$ is negative and $\mathcal{P}_{\sm_{X}}$
is a tree-like semigroup presentation by Remark \ref{rem: psimx}. Suppose that $\ell(w) = [B]$. By the definition of $\mathcal{P}_{\sm_{X}}$, $[B]$ is the left
side of exactly one relation, namely $([B], [B_{1}][B_{2}]\ldots[B_{n}])$, where the $B_{i}$ are maximal proper subballs of $B$, listed in order. It follows
that the bottom label of $T'_{1}$ is $[B_{1}][B_{2}]\ldots[B_{n}]$ and the top label of $T'_{2}$ is $[B_{1}][B_{2}]\ldots[B_{n}]$. Therefore
$(T'_{1}, T'_{2})$ is a dipole. This proves the claim.

A diagram over $\mathcal{P}_{\sm_{X}}$ is \emph{positive} if all of its transistors are positive, and \emph{negative} if all of its transistors are negative.
We note that $\Delta$ is positive if and only if $\Delta^{-1}$ is negative, by the description of inverses in the proof of Theorem \ref{thm: braideddiagramgroups}.
The above reasoning shows that any reduced $([X], [X])$-diagram over $\mathcal{P}_{\sm_{X}}$ can be written
$\Delta = \Delta^{+}_{1} \circ \left( \Delta_{2}^{+} \right)^{-1}$, where $\Delta^{+}_{i}$ is a positive diagram for $i = 1,2$.

We claim that any positive diagram $\Delta$ over $\mathcal{P}_{\sm_{X}}$ with top label $[X]$ is $\Delta_{\mathcal{P}}$ (up to a reordering of the bottom 
contacts), where $\mathcal{P}$ is some partition of $X$. There is a unique wire $W \in \mathcal{W}(\Delta)$ making a top contact with the frame. We call
this wire $W_{X}$. Note that its label is $[X]$ by our assumptions. The bottom contact of $W_{X}$ lies either on the bottom of the frame, or on top of some
transistor. In the first case, we have $\Delta = \Delta_{\mathcal{P}}$ for $\mathcal{P} = \{ X \}$ and we are done. In the second, the bottom contact of 
$W_{X}$ lies on top of some transistor $T$, which we call $T_{X}$. Since the top label of $T_{X}$ is $[X]$, the bottom label must be $[B_{1}]\ldots [B_{k}]$,
where $B_{1}, \ldots, B_{k}$ are the maximal proper subballs of $X$. Thus there are wires $W_{1}, \ldots, W_{k}$ attached to the bottom of $T_{X}$, and we 
have $\ell(W_{i}) = [B_{i}]$, for $i = 1, \ldots, k$. We relabel each of the wires $W_{B_{1}}, \ldots, W_{B_{k}}$, respectively. Note that
$\{ B_{1}, \ldots, B_{k} \}$ is a partition of $X$ into balls. We can continue in this way, inductively labelling each wire with a ball $B \subseteq X$.
If we let $\overline{B}_{1}, \ldots, \overline{B}_{m}$ be the resulting labels of the wires which make bottom contacts with the frame, then 
$\{ \overline{B}_{1}, \ldots, \overline{B}_{m} \} = \mathcal{P}$ is a partition of $X$ into balls, and $\Delta = \Delta_{\mathcal{P}}$ by construction,
up to a reordering of the bottom contacts.

We can now prove surjectivity of $\widehat{\psi}$. Let $\Delta \in D_{b}(\mathcal{P}_{\sm_{X}}, [X])$ be reduced. We can write
$\Delta = \Delta_{2}^{+} \circ \left( \Delta_{1}^{+} \right)^{-1}$, where $\Delta_{i}^{+}$ is positive, for $i=1,2$. It follows that
$\Delta_{i}^{+} = \Delta_{\mathcal{P}_{i}} \circ \sigma_{i}$, for $i=1,2$, where $\mathcal{P}_{i}$ is a partition of $X$ into balls and $\sigma_{i}$ is 
diagram containing no transistors. Thus, $\Delta = \Delta_{\mathcal{P}_{2}} \circ \sigma_{2} \circ \sigma_{1}^{-1} \circ \Delta_{\mathcal{P}_{1}}^{-1}
= \psi(( \mathcal{P}_{1}, \mathcal{P}_{2}, \phi))$, where $\phi: \mathcal{P}_{1} \rightarrow \mathcal{P}_{2}$ is a bijection determined by
$\sigma_{2} \circ \sigma_{1}^{-1}$. Therefore, $\widehat{\psi}$ is surjective.

Now we must show that if $\mathcal{P} = \langle \Sigma \mid \mathcal{R} \rangle$ is a tree-like semigroup presentation, $x \in \Sigma$, then there
is a compact ultrametric space $X_{\mathcal{P}}$, a small similarity structure $\sm_{X_{\mathcal{P}}}$, and a compatible ball order, such that
$D_{b}(\mathcal{P}, x) \cong \Gamma(\sm_{X_{\mathcal{P}}})$. 
Construct a labelled ordered simplicial tree $T_{(\mathcal{P}, x)}$ as follows. Begin with a vertex $\ast$, the root, labelled by $x \in \Sigma$. 
By the definition of tree-like semigroup presentation (Definition \ref{def: treelike}), there is at most  one relation in $\mathcal{R}$ having
the word $x$ as its left side. Let us suppose first that $(x,x_{1}x_{2}\ldots x_{k}) \in \mathcal{R}$, where $k \geq 2$. 
We introduce $k$ children of the root, labelled $x_{1}, \ldots, x_{k}$ (respectively), each connected to the root by an edge. 
The children are ordered from left to right in such a way that we read the
word $x_{1}x_{2}\ldots x_{k}$ as we read the labels of the children from left to right. If, on the other hand, $x$ is not the left side of any relation
in $\mathcal{R}$, then the tree terminates -- there is only the root. We continue similarly: if $x_{i}$ is the left side of some relation 
$(x_{i}, y_{1}y_{2}\ldots y_{m}) \in \mathcal{R}$ ($m \geq 2$),
then this relation is unique and we introduce a labelled ordered collection of children, as above. If $x_{i}$ is not the left side of any relation
in $\mathcal{R}$, then $x_{i}$ has no children. This builds a labelled ordered tree $T_{(\mathcal{P}, x)}$. We note that if a vertex $v 
\in T_{(\mathcal{P}, x)}$ is labelled by $y \in \Sigma$, then the subcomplex $T_{v} \leq T_{(\mathcal{P},x)}$ spanned by $v$ and all of its descendants is isomorphic
to $T_{(\mathcal{P}, y)}$, by a simplicial isomorphism which preserves the labelling and the order. 

We let $\mathrm{Ends}(T_{(\mathcal{P}, x)})$ denote the set of all edge-paths $p$ in $T_{(\mathcal{P}, x)}$ such that:
i) $p$ is without backtracking; ii) $p$ begins at the root; iii) $p$ is either infinite, or $p$ terminates at a vertex without children.
We define a metric on $\mathrm{Ends}(T_{(\mathcal{P}, x)})$ as follows. If $p, p' \in \mathrm{Ends}(T_{(\mathcal{P},  x)})$ and $p, p'$ have exactly
$m$ edges in common, then we set $d(p,p') = e^{-m}$. This metric makes $\mathrm{Ends}(T_{(\mathcal{P},x)})$ a compact ultrametric space, and a ball order is given
by the ordering of the tree. We can describe the
balls in $\mathrm{Ends}(T_{(\mathcal{P},x)})$ explicitly. Let $v$ be a vertex of $T_{(\mathcal{P},x)}$. We set $B_{v} = 
\{ p \in \mathrm{Ends}(T_{(\mathcal{P},x)}) \mid v \text{ lies  on } p \}$. Every such set is a ball, and every ball in $\mathrm{Ends}(T_{(\mathcal{P},x)})$
has this form. We can now describe a finite similarity structure $\sm_{X_{\mathcal{P}}}$ on $\mathrm{Ends}(T_{(\mathcal{P}, x)})$. Let $B_{v}$ and $B_{v'}$ be the balls
corresponding to the vertices $v, v' \in T_{(\mathcal{P},x)}$. If $v$ and $v'$ have different labels, then we set $\sm_{X_{\mathcal{P}}}(B_{v}, B_{v'}) = \emptyset$.
If $v$ and $v'$ have the same label, say $y \in \Sigma$, then there is label- and order-preserving simplicial isomorphism $\psi: T_{v} \rightarrow T_{v'}$.
Suppose that $p_{v}$ is the unique edge-path without backtracking connecting the root to $v$. Any point in $B_{v}$ can be expressed in the form
$p_{v}q$, where $q$ is an edge-path without backtracking in $T_{v}$. We let  $\widehat{\psi}: B_{v} \rightarrow B_{v'}$ be defined by the
rule $\widehat{\psi}(p_{v}q) = p_{v'}\psi(q)$. The map $\widehat{\psi}$ is easily seen to be a surjective similarity. 
We set $\sm_{X_{\mathcal{P}}}(B_{v}, B_{v'}) = \{ \widehat{\psi} \}$. The resulting assignments give a small similarity structure $\sm_{X_{\mathcal{P}}}$ on 
the compact
ultrametric space $\mathrm{Ends}(T_{(\mathcal{P},x)})$ that is compatible with the ball order.

Now we can apply the first part of the theorem: setting $X_{(\mathcal{P},x)} = \mathrm{Ends}(T_{(\mathcal{P},x)})$, we
have $\Gamma(\sm_{X_{(\mathcal{P},x)}}) \cong D_{b} (\mathcal{P}_{\sm_{X_{(\mathcal{P},x)}}}, [X_{(\mathcal{P},x)}]) \cong D_{b}(\mathcal{P}, x)$, where the first isomorphism follows from the forward direction of the theorem, and the second isomorphism follows from the canonical identification of the semigroup presentation $\mathcal{P}_{\sm_{X_{(\mathcal{P},x)}}}$
with $\mathcal{P}$.
\end{proof}



\bibliographystyle{plain}
\bibliography{biblio}

\end{document}